\documentclass[oneside,10pt]{article}
\usepackage[b5paper]{geometry}	

\usepackage{amsfonts,amsmath,latexsym,amssymb}
\usepackage{mathrsfs,upref}
\usepackage{mathptmx}		    	                            		
\usepackage{amscd}
\usepackage{amsthm}
\usepackage{graphicx}
\usepackage{amsmath}
\usepackage{lipsum}

\newcommand\blfootnote[1]{%
  \begingroup
  \renewcommand\thefootnote{}\footnote{#1}%
  \addtocounter{footnote}{-1}%
  \endgroup}

\newtheorem{theorem}{Theorem}

\newtheorem{corollary}[theorem]{Corollary}

\newtheorem{definition}[theorem]{Definition}

\newtheorem{lemma}[theorem]{Lemma}

\newtheorem{proposition}[theorem]{Proposition}
\newtheorem{remark}[theorem]{Remark}

\newcommand\supp{\mathop{\rm supp}}
\newcommand\esssup{\mathop{\rm ess \, sup}}
\newcommand\essinf{\mathop{\rm ess \, inf}}

\begin{document}

\title{Fractional type operators on Hardy spaces associated with ball quasi-Banach function spaces}
\author{Pablo Rocha}

\maketitle

\begin{abstract}
For $0 \leq \alpha < n$ and $m \in \mathbb{N} \cap \left(1 - \frac{\alpha}{n}, +\infty \right)$, we consider certain fractional type 
operators $T_{\alpha, m}$ generated by $m$-orthogonal matrices and prove that, for $0 < \alpha < n$, $T_{\alpha, m}$ can be extended to a bounded operator $H_X \to Y$ and, for $\alpha = 0$, $T_{0, m}$ can be extended to a bounded operator $H_X \to X$, where $X$ and $Y$ are certain ball quasi-Banach spaces related to each other and $H_X$ is the Hardy space associated with $X$. In particular, our results apply to weighted Lebesgue spaces, variable Lebesgue spaces, Lorentz spaces and Orlicz spaces, the last two are new. Our proofs rely on the ssumption that $X$ is $\mathcal{O}(n)$-invariant, the theory of weighted Hardy spaces, the Rubio de Francia iteration algorithm and the finite atomic decomposition of $H_X$.
\end{abstract}


\section{Introduction}

Let $0\leq \alpha <n$ and $m \in \mathbb{N} \cap \left(1 - \frac{\alpha}{n}, +\infty \right)$, we consider the following generalization of the Riesz potential
\begin{equation}
T_{\alpha, m}f(x)=\int_{\mathbb{R}^{n}} \left\vert x-A_{1}y\right\vert ^{-\alpha
_{1}}...\left\vert x-A_{m}y\right\vert ^{-\alpha _{m}}f(y)dy,  \label{T}
\end{equation}
where $\alpha _{1}+...+\alpha _{m}=n-\alpha$, and the $A_j$'s are $n \times n$ orthogonal matrices. For the case $\alpha = 0$, we assume that 
$A_i - A_j$ is invertible if $i \neq j$. Indeed, the family of operators in (\ref{T}) contains to the Riesz potential. For $0 < \alpha < n$, the Riesz potential $I_{\alpha}$ is defined by
\begin{equation} \label{Riesz pot}
I_{\alpha}f(x) = \int_{\mathbb{R}^{n}} \frac{f(y)}{|x-y|^{n- \alpha}} dy, 
\end{equation}
for any $f \in L^{p}(\mathbb{R}^{n})$ and $1 \leq p < \frac{n}{\alpha}$. Then, taking $0 < \alpha < n$, $m=1$ and $A_1 = I$ in (\ref{T}), we have that $I_\alpha = T_{\alpha, 1}$. 

For $\alpha = 0$ and $m \geq 2$, the operator $T_{0,m}$ is bounded on $L^p(\mathbb{R}^n)$ for all $1 < p < \infty$ 
(see \cite[Theorem 6]{Ro-Ur2017}). Now, for $0 < \alpha < n$ and $m \geq 1$, the operator $T_{\alpha, m}$ has the same behavior that the 
Riesz potential on $L^{p}(\mathbb{R}^{n})$. In fact,
\[
|T_{\alpha, m}f(x) | \leq C \sum_{j=1}^{m} \int_{\mathbb{R}^{n}} |A_{j}^{-1}x - y |^{\alpha - n} |f(y)| dy = C \sum_{j=1}^{m} I_{\alpha}(|f|)(A_{j}^{-1}x),
\]
for all $x \in \mathbb{R}^{n}$, this pointwise inequality implies that $T_{\alpha, m}$ is a bounded operator from 
$L^{p}(\mathbb{R}^{n})$ into $L^{q}(\mathbb{R}^{n})$ for $1 < p < \frac{n}{\alpha}$ and $\frac{1}{q} = \frac{1}{p} - \frac{\alpha}{n}$. However, for $0 < \alpha < n$, $m \geq 2$ and $0 < p \leq \frac{n}{n + \alpha}$, the behavior of $T_{\alpha, m}$ on $H^p(\mathbb{R}^n)$ 
differs from that of $I_{\alpha}$. Indeed, it is well known that $I_{\alpha}$ can be extended to 
a bounded operator $H^p(\mathbb{R}^n) \to L^{q}(\mathbb{R}^{n})$ for $0 < p \leq 1$ and $\frac{1}{q} = \frac{1}{p} - \frac{\alpha}{n}$, 
and $H^p(\mathbb{R}^n) \to H^{q}(\mathbb{R}^{n})$ for $0 < p \leq \frac{n}{n+\alpha}$ and $\frac{1}{q} = \frac{1}{p} - \frac{\alpha}{n}$ 
(see \cite{Taible}). On the other hand, the present author together with M. Urciuolo in \cite{Ro-Ur2012} proved that $T_{\alpha, m}$ can be extended to a bounded operator $H^p(\mathbb{R}^n) \to L^{q}(\mathbb{R}^{n})$ for $0 < p \leq 1$ and $\frac{1}{q} = \frac{1}{p} - \frac{\alpha}{n}$, and also showed that the $H^p(\mathbb{R}^n) \to H^{q}(\mathbb{R}^{n})$ boundedness does not 
hold for $T_{\alpha, m}$ with $n=1$, $0 < \alpha < 1$, $m = 2$, $A_1$ = 1, $A_2 = -1$, $0 < p \leq \frac{1}{1+\alpha}$ and 
$\frac{1}{q} = \frac{1}{p} - \alpha$.

The author in \cite{Rocha2023} obtained weighted estimates for the operator $T_{\alpha, m}$. More precisely, for $\alpha = 0$, $m \geq 2$ 
and $0 < p \leq 1$, we proved that the operator $T_{0,m}$ given by (\ref{T}), with $A_i$ real matrices (not necessarily orthogonal) such that 
$A_i - A_j$ invertible if $i \neq j$, can be extended to a bounded operator $H^p_{w}(\mathbb{R}^n) \to L^p_{w}(\mathbb{R}^n)$, when 
$w \in \mathcal{A}_{\infty}$ and, for every $j= 1, ..., m$, $w(A_j x) \lesssim w(x)$ a.e. $x$; and for $0 < \alpha < n$, $m \geq 1$ and 
$0 < s <1$, we proved that $T_{\alpha, m}$ can be extended to a bounded operator $H^p_{w^p}(\mathbb{R}^n) \to L^q_{w^q}(\mathbb{R}^n)$ for every $s \leq p \leq 1$ and $\frac{1}{q} = \frac{1}{p} - \frac{\alpha}{n}$, when $w^{\frac{n}{(n - \alpha)s}} \in \mathcal{A}_1$, 
$\frac{r_w}{r_w - 1} < \frac{n}{\alpha}$ (where $r_w$ is the critical index of $w$ for the reverse H\"older condition) and, 
for every $j= 1, ..., m$, $w(A_j x) \lesssim w(x)$ a.e. $x$.

In the variable setting, considering exponents $p(\cdot)$ satisfying $\log$-H\"older conditions and assuming that $p(A_j x) = p(x)$ for 
each $j=1, ..., m$, the author and M. Urciuolo in \cite{Ro-Urc}, via the infinite atomic decomposition of $H^{p(\cdot)}(\mathbb{R}^n)$ established in \cite{Nakai}, proved the $H^{p(\cdot)}(\mathbb{R}^n) \to L^{q(\cdot)}(\mathbb{R}^n)$ boundedness of $T_{\alpha, m}$ for 
$0 < p_{-} \leq p(\cdot) \leq p_{+} < \frac{\alpha}{n}$ and $\frac{1}{q(\cdot)} = \frac{1}{p(\cdot)} - \frac{\alpha}{n}$, as well as the 
$H^{p(\cdot)}(\mathbb{R}^n) \to H^{q(\cdot)}(\mathbb{R}^n)$ boundedness of the Riesz potential $I_{\alpha}$. Another proof of these results, but by using the finite atomic decomposition developed in \cite{UrWang}, was given by the author in \cite{Rocha2018}, where we rely on
the theory of weighted Hardy spaces and the Rubio de Francia iteration algorithm. Such method allows us to avoid the more delicate convergence arguments that are often necessary when utilizing the infinite atomic decomposition. The boundedness of 
$T_{\alpha, m}$ on variable Morrey-Hardy spaces was obtained by J. Tan and J. Zhao in \cite{Tan2016}. 

Nowadays, many variants of Hardy type spaces on $\mathbb{R}^n$, as weighted Hardy spaces, Hardy-Lorentz spaces, Hardy-Orlicz spaces, 
Hardy-Herz spaces, Hardy-Morrey spaces, Musielak-Orlicz-Hardy spaces, and variable Hardy spaces (among others), meet under a common framework; namely: the theory of Hardy spaces associated with ball quasi-Banach function spaces. This theory was introduced by Y. Sawano, K.-P. Ho, D. Yang and S. Yang in \cite{sawa}. For example, if one considers the space 
$L^{p}(\mathbb{R}^n)$, $0 < p < \infty$, with the usual quasi-norm given by $\| f \|_{p}^p = \int_{\mathbb{R}^n} |f(x)|^p dx$, then 
$X:=(L^{p}(\mathbb{R}^n), \| \cdot \|_{p})$ is a ball quasi-Banach function space and its Hardy type space associated 
$H_X (\mathbb{R}^n)$ coincides with $H^p(\mathbb{R}^n)$, where $H_X (\mathbb{R}^n)$ is defined by (\ref{Hx def}) below and 
$H^p(\mathbb{R}^n)$ is the classical Hardy space defined in \cite{Stein}. 

Ones of the principal results obtained in \cite{sawa} is the atomic and molecular characterization of the Hardy space $H_X (\mathbb{R}^n)$ associated with a ball quasi-Banach function space $X$. These characterizations rely on a number of additional technical and structural assumptions related to the boundedness of the powered Hardy-Littlewood maximal operator on $X$ (see (\ref{A1}) and (\ref{A2}) below). As is well known, in the classic context, such decompositions are very useful when it studies the behavior of certain operators, as singular and fractional integrals, on a given Hardy type space (see for instance \cite{cuerva}, \cite{Nakai}, \cite{Ro-Urc}, \cite{Stein}, 
\cite{wheeden}, \cite{Taible}).

Later, X. Yan, D. Yang and W. Yuan in \cite{Yan} gave a finite atomic characterization of $H_X(\mathbb{R}^n)$. As an
application, they prove that the dual space of $H_X(\mathbb{R}^n)$ is the Campanato space associated with $X$.

Y. Chen, H. Jia and D. Yang in \cite{YChen}, assuming (\ref{A1}) and (\ref{A2}) below, prove that the Riesz potential 
$I_{\alpha}$ can be extended to a bounded operator from $H_X(\mathbb{R}^n)$ to $H_{X^{\beta}}(\mathbb{R}^n)$, $\beta > 1$, if and only in  if 
for any ball $B \subset \mathbb{R}^n$, $|B|^{\frac{\alpha}{n}} \lesssim \| \chi_B \|_{X}^{(\beta-1)/\beta}$, where $X$ is a ball quasi-Banach function space and $X^{\beta}$ denotes the $\beta$-convexification of $X$. Moreover, using extrapolation techniques, the authors also proved the $H_X(\mathbb{R}^n) \to H_Y(\mathbb{R}^n)$ boundedness of $I_{\alpha}$, for certain ball quasi-Banach function spaces $X$ and $Y$.

Recently, the author in \cite{Rocha2025}, for $0 \leq \alpha < n$, studied the convolution operators $T_{\alpha}f := K_{\alpha} \ast f$, 
where $K_{\alpha}$ is a kernel of type $(\alpha, N)$ (see also \cite{Folland}). These operators include the classical singular and 
fractional integrals. By means of the infinite atomic decomposition and the maximal characterization of $H_X(\mathbb{R}^n)$ established in 
\cite{sawa}, together with an off-diagonal Fefferman-Stein vector-valued inequality for the fractional maximal operator on the 
$p$-convexification of ball quasi-Banach function spaces, we prove that if $X$ and $Y$ are ball quasi-Banach function spaces satisfying certain hypotheses and $N$ is conveniently chosen, then, for $\alpha = 0$, the operator $T_0$ can be extended to a bounded operator 
$H_{X}(\mathbb{R}^n) \to X$ and $H_{X}(\mathbb{R}^n) \to H_{X}(\mathbb{R}^n)$; and, for $0 < \alpha < n$, the operator $T_{\alpha}$ can 
be extended to a bounded operator $H_{X}(\mathbb{R}^n) \to Y$ and $H_{X}(\mathbb{R}^n) \to H_{Y}(\mathbb{R}^n)$. These results apply to weighted Lebesgue spaces, variable Lebesgue spaces, mixed-norm Lebesgue spaces and Lorentz spaces.

We point out that for $0 \leq \alpha < n$ and $m \geq 2$, the operator $T_{\alpha, m}$ given by (\ref{T}) is not a convolution type operator.

Our main result is established in Theorem \ref{main thm} below, which state that, for $0 < \alpha < n$, $T_{\alpha, m}$ can be extended to a bounded operator $H_X \to Y$ and, for $\alpha = 0$, $T_{0, m}$ can be extended to a bounded operator $H_X \to X$, where $X$ and $Y$ are certain ball quasi-Banach spaces related to each other and $H_X$ is the Hardy space associated with $X$. In particular, our results apply to weighted Lebesgue spaces, variable Lebesgue spaces, Lorentz spaces and Orlicz spaces, the last two are new. Our proofs rely on the assumption that $X$ is $\mathcal{O}(n)$-invariant, the theory of weighted Hardy spaces, the Rubio de Francia iteration algorithm and the finite atomic decomposition of $H_X$.

\

The paper is organized as follows. Section \ref{basico} starts with the basics of the Hardy spaces theory associated with ball quasi-Banach function spaces. The finite atomic Hardy space associated with a ball quasi-Banach function space is also presented in Section \ref{basico}. Some weighted Hardy estimates for the operator $T_{\alpha, m}$ are established in Section \ref{weighted estimates}. In Section 
\ref{resultados principales}, we state and prove our main result. Finally, in Section \ref{appl}, we give four concrete applications to illustrate our results.

\

\textbf{Notation:} The symbol $S \lesssim T$ stands for the inequality $S \leq cT$ for some constant $c$. The symbol $S \approx T$ 
stands for $T \lesssim S \lesssim T$. We denote by $B(x_0, r)$ the ball centered at $x_0 \in \mathbb{R}^{n}$ of radius $r$. Given 
$\gamma >0$ and a ball $B = B(x_0, r)$, we set $\gamma B = B(x_0, \gamma r)$. The orthogonal group $\mathcal{O}(n)$ is defined by 
$\mathcal{O}(n):=\{ A \in GL_n(\mathbb{R}) : A^{t}= A^{-1} \}$. For a measurable subset $E \subset \mathbb{R}^{n}$ we denote $|E|$ and 
$\chi_E$ the Lebesgue measure of $E$ and the characteristic function of $E$ respectively. Given a real number $s \geq 0$, we write 
$\lfloor s \rfloor$ for the integer part of $s$. As usual we denote with $\mathcal{S}(\mathbb{R}^{n})$ the space of smooth and rapidly decreasing functions, with $\mathcal{S}'(\mathbb{R}^{n})$  the dual space. If $\beta$ is the multiindex $\beta=(\beta_1, ..., \beta_n)$, then 
$|\beta| = \beta_1 + ... + \beta_n$. Given a function $g$ on $\mathbb{R}^n$ and $t > 0$, we write $g_t(x) = t^{-n} g(t^{-1} x)$. Let $d$ be a non negative integer and $\mathcal{P}_{d}$ the subspace of $L^{1}_{loc}(\mathbb{R}^{n})$ formed by all the polynomials of degree at most $d$. Given a measurable function $h$, the expression $h \perp \mathcal{P}_{d}$ stands for $\int h(x) P(x) dx = 0$ for all $P \in \mathcal{P}_{d}$.

Throughout this paper, $C$ will denote a positive constant, not necessarily the same at each occurrence.

\section{Preliminaries} \label{basico}

\subsection{Ball quasi-Banach function spaces} In the sequel, $\mathfrak{M} = \mathfrak{M}(\mathbb{R}^n)$ is the set of all measurable functions on $\mathbb{R}^n$ and $\mathfrak{M}_{+} = \mathfrak{M}_{+}(\mathbb{R}^n)$ is the cone of all non-negative measurable functions on $\mathbb{R}^n$.

For every $x \in \mathbb{R}^n$ and $r > 0$ fixed, let $B(x,r) := \{ y \in \mathbb{R}^n: |x-y| < r \}$. Now, we define
\[
\mathbb{B} := \{ B(x,r) : x \in \mathbb{R}^n \,\, \text{and} \,\, r > 0 \}.
\]

\begin{definition}
A mapping $\rho : \mathfrak{M}_{+} \to [0, \infty]$ is called a ball quasi-Banach function norm if it satisfy the following properties:

$(P1)$ $\rho(f) = 0$ implies that $f = 0$ a.e.; 

$(P2)$ $\rho(\alpha f) = |\alpha| \rho(f)$ for all $\alpha \in \mathbb{C}$ and all $f \in \mathfrak{M}_{+}$;

$(P3)$ there exists $C \geq 1$ such that $\rho(f+g) \leq C (\rho(f) + \rho(g))$ for all $f, g \in \mathfrak{M}_{+}$;

$(P4)$ if some $f, g \in \mathfrak{M}_{+}$ satisfy $f \leq g$ a.e., then $\rho(f) \leq \rho(g)$;

$(P5)$ if some $f_n, f \in \mathfrak{M}_{+}$ satisfy $f_n \uparrow f$ a.e., then $\rho(f_n) \uparrow \rho(f)$;

$(P6)$ if $B \in \mathbb{B}$, then $\rho(\chi_B) < \infty$.
\end{definition}

\begin{definition}
Let $\rho$ be a ball quasi-Banach function norm. We then define the corresponding ball quasi-Banach function space $X = X(\rho)$ as the set
\[
X = \{ f \in \mathfrak{M} : \rho(|f|) < \infty \}.
\]
For each $f \in X$, define
\[
\Vert f \Vert_X = \rho(|f|).
\]
\end{definition}

\begin{remark} \label{prop X}
Let $(X, \Vert \cdot \Vert_X)$ be a ball quasi-Banach function space, then it is easy to check that

$(i)$ $\Vert f \Vert_X = 0$ implies that $f = 0$ a.e.;

$(ii)$ $\Vert \alpha f \Vert_X = |\alpha| \Vert f \Vert_X$ for all $\alpha \in \mathbb{C}$ and all $f \in X$;

$(iii)$ there exists $C \geq 1$ such that $\Vert f+g \Vert_X \leq C (\Vert f \Vert_X + \Vert g \Vert_X)$ for all $f, g \in X$;

$(iv)$ if $f \in \mathfrak{M}$, $g \in X$ are such that $|f| \leq |g|$ a.e., then $f \in X$ and $\Vert f \Vert_X \leq \Vert g \Vert_X$;

$(v)$ if $0 \leq f_n \uparrow f$ a.e., then either $f \notin X$ and $\Vert f_n \Vert_X \uparrow \infty$, or $f \in X$ and 
$\Vert f_n \Vert_X \uparrow \Vert f \Vert_X$;

$(vi)$ if $B \in \mathbb{B}$, then $\chi_B \in X$.
\end{remark}

A ball quasi-Banach function space $X$ is called a \textit{ball Banach function space} if the constant $C$, appearing in $(iii)$ of 
Remark \ref{prop X}, is equal to $1$, and for any ball $B \in \mathbb{B}$ there exists a positive constant $C_{(B)}$, depending on $B$, such that
\[
\int_{B} |f(x)| dx \leq C_{(B)} \| f \|_{X},
\]
for all $f \in X$. Thus, every Banach function space is a ball Banach function space (see \cite[Definition 1.3]{Bennett}). An interesting work on the properties of (quasi)-Banach functions spaces can be found in \cite{Nekvinda}.

\begin{definition} (See \cite{Bennett})
For any ball Banach function space $X$, the associate space (also called the K\"othe dual) $X'$ is defined by setting
\[
X' := \left\{ f \in \mathfrak{M} : \| f \|_{X'} := \sup \{ \| f g\|_{L^{1}(\mathbb{R}^n)} : g \in X, \| g \|_{X} = 1 \} < \infty \right\}
\]
where $\| \cdot \|_{X'}$ is called the associate norm of $\| \cdot \|_{X}$.
\end{definition}

\begin{lemma} (\cite[Proposition 2.3]{sawa} and \cite[Lemma 2.6]{Zhang}) \label{doble dual}
Let $X$ be a ball Banach function space. Then $X'$ is also a ball Banach function space and $X$ coincides with its second associate space 
$X''$. In other words, a function $f$ belongs to $X$ if and only if it belongs to $X''$ and, in that case,
\[
\| f \|_{X} = \| f \|_{X''}.
\]
\end{lemma}

\begin{definition}
A function $f \in X$ is said to have absolutely continuous quasi-norm in $X$ if $\| f \chi_{E_j} \|_{X} \downarrow 0$ for every sequence 
$\{ E_j \}_{j=1}^{\infty}$ satisfying $E_j \supset E_{j+1}$ for all $j \in \mathbb{N}$ and $\bigcap_{j=1}^{\infty} E_j = \emptyset$. The set of all functions in $X$ of absolutely continuous quasi-norm is denoted by $X_{a}$. If $X=X_a$, then the space $X$ itself is said to have an absolutely continuous quasi-norm.
\end{definition}

Taking into account that the orthogonal group $\mathcal{O}(n)$ induces an action on functions by $f_{A}(x) = f(A^{-1}x)$, where $A \in \mathcal{O}(n)$, we introduce the following definition.

\begin{definition}
We say that a ball quasi-Banach function space $X$ is $\mathcal{O}(n)$-invariant if for every $f \in X$ and 
$A \in \mathcal{O}(n)$, $f_A \in X$ with $\| f_A \|_{X} =  \| f \|_{X}$.
\end{definition}

\begin{remark} \label{On 1}
It is easy to check that if $X$ is a $\mathcal{O}(n)$-invariant ball Banach function space, then $X'$ is $\mathcal{O}(n)$-invariant.
\end{remark}

\begin{remark}
We point out that not all ball quasi-Banach function space is $\mathcal{O}(n)$-invariant. For instance, the mixed-norm spaces are not 
$\mathcal{O}(n)$-invariant (see \cite{Benedek}).
\end{remark}

\begin{definition}
Let $X$ be a ball quasi-Banach function space and $p \in (0, \infty)$. The $p$-convexification $X^p$ of $X$ is defined by setting 
$X^p := \{ f \in \mathfrak{M} : |f|^p \in X \}$ equipped with the quasi-norm $\| f \|_{X^p} : =  \| |f|^p \|^{1/p}_{X}$.
\end{definition}

\begin{remark} \label{On 2}
If $X$ is a $\mathcal{O}(n)$-invariant ball quasi-Banach function space and $p \in (0, \infty)$, then $X^p$ is $\mathcal{O}(n)$-invariant.
\end{remark}

\begin{definition}
Let $X$ be a ball quasi-Banach function space and $p \in (0, \infty)$. The space $X$ is said to be $p$-convex if there exists a positive constant $C$ such that, for any $\{ f_j \}_{j=1}^{\infty} \subset X^{1/p}$,
\[
\left\| \sum_{j=1}^{\infty} |f_j| \right\|_{X^{1/p}} \leq C \sum_{j=1}^{\infty} \left\| f_j \right\|_{X^{1/p}}.
\]
In particular, when $C = 1$, $X$ is said to be strictly $p$-convex.
\end{definition}

\subsection{Maximal functions and additional assumptions} \label{hypotheses} For $0 \leq \alpha < n$, we define the \textit{fractional maximal operator} 
$M_{\alpha}$ by
\[
(M_{\alpha}f)(x) = \sup_{B \ni x} |B|^{\frac{\alpha}{n} - 1}\int_{B} |f(y)| \, dy,
\]
where $f$ is a locally integrable function on $\mathbb{R}^{n}$ and the supremum is taken over all balls $B$ containing $x$. For 
$\alpha=0$, we have that $M_0 = M$, where $M$ is the {\it Hardy-Littlewood maximal operator} on $\mathbb{R}^{n}$.

For $\theta \in (0, \infty)$, the \textit{powered Hardy-Littlewood maximal operator} $M^{(\theta)}$ is defined by
\[
(M^{(\theta)}f)(x) = \left[ M (|f|^{\theta})(x) \right]^{1/ \theta}.
\]

\begin{remark} \label{Jensen}
Let $0 < \gamma < \theta < \infty$. Then, by Jensen's inequality, one has
\[
(M^{(\gamma)}f)(x) \leq (M^{(\theta)}f)(x), \,\,\,\, \forall x \in \mathbb{R}^n.
\]
\end{remark}

In what follows, we will assume the following two additional assumptions:

\

$A1)$ Let $X$ be a ball quasi-Banach function space. Assume that, for some $\theta, s \in (0, 1]$ with $\theta < s$, there exists a positive constant $C$ such that for any sequence of functions $\{ f_j \}_{j=1}^{\infty} \subset L^{1}_{loc}(\mathbb{R}^n)$
\begin{equation} \label{A1}
\left\Vert  \left\{ \sum_{j=1}^{\infty} (M^{(\theta)} f_j)^{s} \right\}^{1/s} \right\Vert_{X} \leq C
\left\Vert \left\{ \sum_{j=1}^{\infty} |f_j|^{s} \right\}^{1/s}  \right\Vert_{X}.
\end{equation}

$A2)$ Let $X$ be a ball quasi-Banach function space. Assume that there exist $s \in (0, 1]$, $q \in (1, \infty]$ and a positive constant $C$ such that $X^{1/s}$ is a ball Banach function space and for any $f \in (X^{1/s})'$
\begin{equation} \label{A2}
\left\Vert  M^{((q/s)')} f \right\Vert_{(X^{1/s})'} \leq C \left\Vert f  \right\Vert_{(X^{1/s})'}.
\end{equation}

\begin{remark} \label{cond equiv A2}
We observe that (\ref{A2}) is equivalent to that the Hardy-Littlewood maximal operator $M$ be bounded on $[(X^{1/s})']^{1/(q/s)'}$.
\end{remark}

\subsection{Hardy spaces associated with ball quasi-Banach function spaces} Let $(X, \Vert \cdot \Vert_X)$ be a ball quasi-Banach function space. Now, following to \cite{sawa}, we introduce the Hardy type space associated with $X$, which is denoted by $H_{X}(\mathbb{R}^n)$. We also present the finite Hardy space $H^{X, q, d}_{fin}(\mathbb{R}^n)$ given in \cite{Yan}.

Given $N \in \mathbb{Z}_{+}$, let 
\[
\mathcal{F}_{N}=\left\{ \varphi \in \mathcal{S}(\mathbb{R}^{n}):\sum\limits_{\left\vert \mathbf{\beta }\right\vert \leq N}\sup\limits_{x\in \mathbb{R}^{n}}\left( 1+\left\vert x\right\vert \right)^{N}\left\vert \partial^{\mathbf{\beta }}
\varphi(x) \right\vert := \| \varphi \|_{\mathcal{S}(\mathbb{R}^{n}), \, N}  \leq 1\right\}.
\] 
Given $f \in \mathcal{S}'(\mathbb{R}^{n})$, we define the grand maximal function of $f$ as 

\begin{equation} \label{grand max 0}
\mathcal{M}^{0}_{N} f(x) :=\sup \left\{ |\left( \phi_t \ast f\right)(x) | : t > 0, \phi \in \mathcal{F}_{N} \right\}.
\end{equation}

\begin{definition} (\cite[Theorem 3.1 - (ii) with $b = n/r + 1$]{sawa})
Let $X$ be a ball quasi-Banach function space such that the Hardy-Littlewood maximal operator $M$ is bounded on $X^{1/r}$ for some 
$r \in (0, \infty)$. For $N \geq \lfloor  n/r + 3 \rfloor$, define the Hardy space $H_{X}(\mathbb{R}^n)$ associated with $X$ as
\begin{equation}  \label{Hx def}
H_{X}(\mathbb{R}^n) = \left\{ f \in \mathcal{S}'(\mathbb{R}^n) : \| \mathcal{M}^{0}_{N} f \|_{X} < \infty \right\},
\end{equation}
Fixed  $N \geq \lfloor  n/r + 3 \rfloor$, we consider  $\| f \|_{H_{X}(\mathbb{R}^n)} := \| \mathcal{M}^{0}_{N} f \|_{X}$.
\end{definition}

Before establishing the finite Hardy space $H^{X, q, s}_{fin}(\mathbb{R}^n)$, we recall the definition of $X$-atom.

\begin{definition} \label{atom}
Let $X$ be a ball quasi-Banach function space satisfying (\ref{A1}) for some $0 < \theta < s \leq 1$ and (\ref{A2}) for some 
$q \in (1, \infty]$ and the same $s$ as in (\ref{A1}). Assume that $d \in \mathbb{Z}_{+}$ satisfies 
$d \geq d_X$, where $d_X := \lfloor n (1/\theta - 1) \rfloor$ and $\theta \in (0, 1]$ is the constant in (\ref{A1}). Then a function 
$a(\cdot)$ is called an $(X, q, d)$-atom if there exists a ball $B \in \mathbb{B}$ such that $\supp(a) \subset B$,
\[
\| a \|_{L^q(\mathbb{R}^n)} \leq \frac{|B|^{1/q}}{\| \chi_B \|_{X}},
\]
and $a(\cdot) \perp \mathcal{P}_d$.
\end{definition}

\begin{remark} \label{infinite atom}
For $q \geq 1$ fixed, every $(X, \infty, d)$-atom is an $(X, q, d)$-atom.
\end{remark}

\begin{definition}
Let $X$, $q$, $d$ and $s$ be as in Definition \ref{atom}. The finite atomic Hardy space $H^{X, q, d}_{fin}(\mathbb{R}^n)$, 
associated to $X$, is defined to be the set of all finite linear combinations of $(X, q, d)$-atoms. The quasi-norm 
$\| \cdot \|_{H^{X, q, d}_{fin}(\mathbb{R}^n)}$ in $H^{X, q, d}_{fin}(\mathbb{R}^n)$ is defined, for any 
$f \in H^{X, q, d}_{fin}(\mathbb{R}^n)$, by setting
\[
\| f \|_{H^{X, q, d}_{fin}(\mathbb{R}^n)} := \inf \left\{ \left\| \left\{ \sum_{j=1}^{N} \left( \frac{\lambda_j}{\| \chi_{B_j} \|_X} 
\right)^{s} \chi_{B_j} \right\}^{1/s} \right\|_X : f = \sum_{j=1}^{N} \lambda_j a_{j}, \, \{ \lambda_j \}_{j=1}^{N} \subset [0, \infty) \right\},
\]
where the infimum is taken over all finite linear combinations of $f$ in terms of $(X, q, d)$-atoms $\{ a_j \}_{j=1}^{N}$ supported, respectively, in the balls $\{ B_j \}_{j=1}^{N}$.
\end{definition}

\begin{theorem} \cite[Theorem 1.10]{Yan} \label{equiv quasi-norm}
Let $X$, $q$, and $d$ be as in Definition \ref{atom}. If $q \in (1, \infty)$, then 
$H^{X, q, d}_{fin}(\mathbb{R}^n) \subset H_{X}(\mathbb{R}^n)$. Moreover, $\| \cdot \|_{H^{X, q, d}_{fin}(\mathbb{R}^n)}$ 
and $\| \cdot \|_{H_{X}(\mathbb{R}^n)}$ are equivalent quasi-norms on the space $H^{X, q, d}_{fin}(\mathbb{R}^n)$.
\end{theorem}

By \cite[Theorems 3.6 and 3.7]{sawa} (which also hold with balls instead of cubes), and \cite[Corollary 3.11 - (ii)]{sawa}, we obtain the following result.

\begin{corollary} \label{dense set}
Let $s$, $q$ and $d$ be as in Definition \ref{atom} and let $X$ be a strictly $s$-convex ball quasi-Banach function spaces such that the quasi-norm of $X$ is absolutely continuous. Then $H^{X, q, d}_{fin}(\mathbb{R}^n) \subset H_{X}(\mathbb{R}^n)$ densely. 
\end{corollary}

\section{Estimates on weighted Hardy type spaces} \label{weighted estimates}

A weight is a non-negative locally integrable function on $\mathbb{R}^{n}$ that takes values in $(0, \infty)$ almost everywhere, i.e. : the weights are allowed to be zero or infinity only on a set of Lebesgue measure zero.

Given a weight $w$ and $0 < p < \infty$, we denote by $L^{p}(w)$ the spaces of all functions $f$ defined on $\mathbb{R}^n$ satisfying 
$\| f \|_{L^{p}(w)} := (\int_{\mathbb{R}^{n}} |f(x)|^{p} w(x) dx)^{1/p} < \infty$ . When $p=\infty$, we have that 
$L^{\infty}(w) = L^{\infty}(\mathbb{R}^{n})$ with $\| f \|_{L^{\infty}(w)} := \| f \|_{L^{\infty}}$. If $E$ is a measurable set, we use the notation $w(E) = \int_{E} w(x) dx$.

We say that a weight $w \in \mathcal{A}_1$ if there exists $C >  0$ such that
\begin{equation*}
M(w)(x) \leq C w(x), \,\,\,\,\, a.e. \, x \in \mathbb{R}^{n}, 
\end{equation*}
the best possible constant is denoted by $[w]_{\mathcal{A}_1}$. Equivalently, a weight $w \in \mathcal{A}_1$ if there exists $C >  0$ such that for every ball $B$
\begin{equation}
\frac{1}{|B|} \int_{B} w(x) dx \leq C \, ess\inf_{x \in B} w(x). \label{A1condequiv}
\end{equation}

\begin{remark} \label{wr A1}
If $w \in \mathcal{A}_1$ and $0 < r < 1$, then by H\"older inequality we have that $w^{r} \in \mathcal{A}_1$.
\end{remark}

\begin{remark} \label{Dn}
If $w \in \mathcal{A}_1$, then 
\[
w(B(x,tr)) \leq [w]_{\mathcal{A}_1} t^n w(B(x,r)), \,\,\, \forall t \geq 1.
\]
So, $w$ satisfies the doubling $D_n$ condition (see \cite[p. 2, (vii)]{tor}).
\end{remark}

\begin{remark} \label{O de n}
The orthogonal group $\mathcal{O}(n)$ induces an action on functions by $f_{A}(x) = f(A^{-1}x)$, where $A \in \mathcal{O}(n)$.
It is easy to check that $w \in \mathcal{A}_1$ if and only if $w_{A} \in \mathcal{A}_1$ for all $A \in \mathcal{O}(n)$. Therefore, the space of weights $\mathcal{A}_1$ is preserved by the action of $\mathcal{O}(n)$.
\end{remark}

For $1 < p < \infty$, we say that a weight $w \in \mathcal{A}_p$ if there exists $C> 0$ such that for every ball $B$
\[
\left( \frac{1}{|B|} \int_{B} w(x) dx \right) \left( \frac{1}{|B|} \int_{B} [w(x)]^{-\frac{1}{p-1}} dx \right)^{p-1} \leq C.
\]
It is well known that $\mathcal{A}_{p_1} \subset \mathcal{A}_{p_2}$ for all $1 \leq p_1 < p_2 < \infty$. 

Given $1 < p \leq q < \infty$, we say that a weight $w \in \mathcal{A}_{p,q}$ if there exists $C> 0$ such that for every ball $B$
\[
\left( \frac{1}{|B|} \int_{B} [w(x)]^{q} dx \right)^{1/q} \left( \frac{1}{|B|} \int_{B} [w(x)]^{-p'} dx \right)^{1/p'} \leq C < \infty.
\]
For $p=1$, we say that a weight $w \in \mathcal{A}_{1,q}$ if there exists $C> 0$ such that for every ball $B$
\[
\left( \frac{1}{|B|} \int_{B} [w(x)]^{q} dx \right)^{1/q} \leq C \, ess\inf_{x \in B} w(x).
\]
When $p=q$, this definition is equivalent to $w^{p} \in \mathcal{A}_{p}$.

A weight $w$ satisfies the reverse H\"older inequality with exponent $s > 1$, denoted by $w \in RH_{s}$, if there exists $C> 0$ such that for every ball $B$,
\[
\left(\frac{1}{|B|} \int_{B} [w(x)]^{s} dx \right)^{\frac{1}{s}} \leq C \frac{1}{|B|} \int_{B} w(x) dx;
\]
the best possible constant is denoted by $[w]_{RH_s}$. 

\begin{remark} \label{RHs RHt}
We observe that if $w \in RH_s$, then by H\"older's inequality, $w \in RH_t$ for all $1 < t < s$, and
$[w]_{RH_t} \leq [w]_{RH_s}$.
\end{remark}

The following result was proved for cubes in \cite{Lerner}. However, since $w \in \mathcal{A}_1$ is doubling, the Lemma holds for balls with the same exponent.

\begin{lemma} \label{A1 y RHs}
Given $w \in \mathcal{A}_1$, then $w \in RH_s$, where $s= 1 + (2^{n+1} [w]_{\mathcal{A}_1})^{-1}$.
\end{lemma}

Given a weight $w \in \mathcal{A}_{1}$ and $0 < p \leq 1$, the weighted Hardy space $H^{p}(w)$ consists of all tempered distributions $f$ on 
$\mathbb{R}^n$ such that
\[
\| f \|_{H^{p}(w)} := \| \mathcal{M}^{0}_{N}f \|_{L^{p}(w)} = \left( \int_{\mathbb{R}^{n}}  
[\mathcal{M}^{0}_{N} f(x)]^{p} w(x) dx \right)^{1/p} < \infty,
\]
where $\mathcal{M}^{0}_{N}$ is the grand maximal given in (\ref{grand max 0}) and $N \geq \lfloor n (p^{-1}-1) \rfloor$.

\

Next, we give an atomic reconstruction in weighted Hardy spaces considering $X$-atoms.

\begin{lemma} \label{X atom Hw}
Let $X$ be an ball quasi-Banach function space and $0 < p_0 \leq 1$. Suppose $\{ a_j \}_{j=1}^{\infty}$ is a sequence of 
$(X, q, d)$-atoms supported, respectively, in the balls $\{ B_j \}_{j=1}^{\infty}$, $\{ \lambda_j \}_{j=1}^{\infty}$ is a non-negative sequence, and $w \in \mathcal{A}_1 \cap RH_{(q/p_0)'}$. If
\[
\left\|  \sum_{j=1}^{\infty} \frac{\lambda_j}{\| \chi_{B_j} \|_{X}} \chi_{B_j} \right\|_{L^{p_0}(w)} < \infty,
\]
then the series $f = \sum_j \lambda_j a_j$ converges in $H^{p_0}(w)$, and
\[
\| f \|_{H^{p_0}(w)} \leq C \left\|  \sum_{j=1}^{\infty} \frac{\lambda_j}{\| \chi_{B_j} \|_{X}} \chi_{B_j} \right\|_{L^{p_0}(w)},
\]
where the constant $C$ does not depend on $f$, $\{ a_j \}_{j=1}^{\infty}$ and $\{ \lambda_j \}_{j=1}^{\infty}$.
\end{lemma}

\begin{proof}
Since $w \in \mathcal{A}_1$, by Remark \ref{Dn}, $w$ satisfies the doubling $D_n$ condition. Now, by Remark \ref{Jensen}, the constant 
$\theta$ that appears in Definition \ref{atom} can be taken such that 
\[
\lfloor n (\theta^{-1} - 1) \rfloor \geq \lfloor n (p_0^{-1} - 1) \rfloor.
\]
Finally, we apply \cite[Theorem 1, on p. 111-112]{tor} with $s=1$, 
\[
\widetilde{a}_j = \| \chi_{B_j} \|_{X} \, a_j, \,\,\,\, \text{and} \,\,\,\, \widetilde{\lambda}_j = \| \chi_{B_j} \|_{X}^{-1} \lambda_j.
\]
This concludes the proof.
\end{proof}

Now, we introduce the set $H^{X, p_0, q, d}_{fin}(w)$, which will be used as a bridge to achieve our main result of Section \ref{resultados principales}. This set is similar to that considered in \cite[Section 7.3]{UrWang}.

\begin{definition}
Let $X$ be a ball quasi-Banach function space, $0 < p_0 \leq 1$, and let $q$ and $d$ be as in Definition \ref{atom}. Given 
$w \in \mathcal{A}_1$, define $H^{X, p_0, q, d}_{fin}(w)$ to be the set of all finite sums of $(X, q, d)$-atoms. For any 
$f \in H^{X, p_0, q, d}_{fin}(w)$, we set
\[
\| f \|_{H^{X, p_0, q, d}_{fin}(w)} := \inf \left\{ \left\| \sum_{j=1}^{N} \left( \frac{\lambda_j}{\| \chi_{B_j} \|_X} 
\right)^{p_0} \chi_{B_j} \right\|_{L^1(w)}^{1/p_0} : f = \sum_{j=1}^{N} \lambda_j a_{j}, \, \{ \lambda_j \}_{j=1}^{N} \subset [0, \infty) \right\},
\]
where the infimum is taken over all finite linear combinations of $f$ in terms of $(X, q, d)$-atoms $\{ a_j \}_{j=1}^{N}$ supported, respectively, in the balls $\{ B_j \}_{j=1}^{N}$.
\end{definition}

From this definition, it is clear that $H^{X, p_0, q, d}_{fin}(w) = H^{X, q, d}_{fin}(\mathbb{R}^n)$ as sets.

\begin{lemma} \label{HX finito 2}
Let $X$ be a ball quasi-Banach function space such that $X^{1/p_0}$ is a ball Banach function space for some $0 < p_0 \leq 1$, and 
let $q$ and $d$ be as in Definition \ref{atom}. If $w \in \mathcal{A}_1 \cap RH_{(q/p_0)'} \cap (X^{1/p_0})'$, then 
$H^{X, p_0, q, d}_{fin}(w) \subset H^{p_0}(w)$ and
\[
\| f \|_{H^{X, p_0, q, d}_{fin}(w)} \approx \| f \|_{H^{p_0}(w)},
\]
for all $f \in H^{X, p_0, q, d}_{fin}(w)$.
\end{lemma}

\begin{proof}
From Lemma \ref{X atom Hw}, it follows that $H^{X, p_0, q, d}_{fin}(w) \subset H^{p_0}(w)$, and there exists an universal constant 
$C > 0$ such that
\[ 
\| f \|_{H^{p_0}(w)} \leq C \| f \|_{H^{X, p_0, q, d}_{fin}(w)},
\]
for all $f \in H^{X, p_0, q, d}_{fin}(w)$. 

To obtain the opposite inequality, fix $f \in H^{X, p_0, q, d}_{fin}(w)$, then $f$ is supported on a ball $B = B(0, K)$ for some $K > 1$. 
Let $\widetilde{B} = B(0, 4K)$. We assume $\| f \|_{H^{p_0}(w)} = 1$, and observe that the space $L^{p_0}(w)$, with $w \in \mathcal{A}_1$, is a ball quasi-Banach function space such that the assumption (\ref{A1}) holds true for any $\theta, s \in (0,1]$, $\theta < s$, 
$p_0 \in (\theta, 1]$ and $w \in \mathcal{A}_{p_0/\theta}$ (it follows from \cite[Theorem 3.1 (b)]{John} applied on $L^{p_0/\theta}(w)$, with 
$r = s/\theta$ and $|f_j|^{\theta}$ instead of $f_k$). So, we can apply \cite[Lemma 2.4]{Yan} with $\widetilde{X} = L^{p_0}(w)$ to obtain 
\begin{equation} \label{pointwise grand max}
\mathcal{M}_{N}^{0} f(x) \leq C \| \chi_B \|_{L^{p_0}(w)}^{-1} = C w(B)^{-1/p_0}, \,\,\,\, \forall \, x \in \mathbb{R}^n \setminus B(0, 4K).
\end{equation}
Now, we will prove that there exists an universal constant $C>0$ such that  $\| f \|_{H^{X, p_0, q, d}_{fin}(w)} \leq C$, under the assumption 
$\| f \|_{H^{p_0}(w)} = 1$, and thus the proposition will be followed from the homogeneity of $\| \cdot \|_{H^{X, p_0, q, d}_{fin}(w)}$. Since this estimate is similar to the proof of \cite[Lemma 7.11]{UrWang}, we will give the main steps of its proof adapting them to our context. We have, by the proof of \cite[Theorem 1, on p. 118]{tor}, that
\[
f = \sum_{j,k}^{\infty} \lambda_{j,k} \| \chi_{B_{j,k}} \|_{X}^{-1} \| \chi_{B_{j,k}} \|_{X} a_{j,k} = \sum_{j,k}^{\infty} \lambda_{j,k} 
a_{j,k},
\]
where every $a_{j,k}$ is an $(X, \infty, d_X)$-atom supported on the ball $B_{j,k}$ (we recall that, by Remark \ref{Jensen}, $d_X$ can be 
taken such that $d_X \geq \lfloor n (p_0^{-1} - 1) \rfloor$, and so $\widetilde{a}_{j,k} := \| \chi_{B_{j,k}} \|_{X} \, a_{j,k}$ is 
an $(\infty, d_X)$-atom in the sense of \cite[p. 111]{tor}), the $\lambda_{j,k}$'s are non-negative scalars, and
\begin{equation} \label{ineq pesada}
\left\| \sum_{j,k}^{\infty} \left( \frac{\lambda_{j,k}}{\| \chi_{B_{j,k}} \|_{X}} \right)^{p_0} \chi_{B_{j,k}} \right\|_{L^1(w)}^{1/p_0} \leq
C \| f \|_{H^{p_0}(w)} = C.
\end{equation}
Therefore, by taking into account (\ref{pointwise grand max}), we can write (as in \cite[p. 484]{UrWang})
\[
f = h + \ell = k_{w}^{-1} (k_w h) + (\ell - \ell_i) + \sum_{F_i} \lambda_{j,k} a_{j,k}, 
\]
where $k_w$ is a constant given by
\[
k_w = \frac{c^{-1} w(B)^{1/p_0}}{\| \chi_{\widetilde{B}} \|_{X}},
\] 
$k_w h$ is an $X$-atom, $i$ is chosen large enough so that $(\ell - \ell_i)$ is  an $X$-atom, and the last sum is a finite sum of $X$-atoms. 
Then,
\[
\| f \|_{H^{X, p_0, q, d}_{fin}(w)}^{p_0} \leq k_{w}^{-p_0} \frac{w(\widetilde{B})}{\| \chi_{\widetilde{B}} \|_{X}^{p_0}} + 
\frac{w(\widetilde{B})}{\| \chi_{\widetilde{B}} \|_{X}^{p_0}} +
\left\| \sum_{j,k}^{\infty} \left( \frac{\lambda_{j,k}}{\| \chi_{B_{j,k}} \|_{X}} \right)^{p_0} \chi_{B_{j,k}} \right\|_{L^1(w)}.
\]
By Remark \ref{Dn}, the first term is bounded by a constant. By (\ref{ineq pesada}), the last term is bounded by the constant 
$C^{p_0}$. Finally, to bound the middle term, we apply \cite[Lemma 2.5]{Zhang}, with $X^{1/p_0}$, and obtain
\[
w(\widetilde{B}) = \int_{\widetilde{B}} w(x) dx \leq C \| \chi_{\widetilde{B}} \|_{X^{1/p_0}} \| w \|_{(X^{1/p_0})'} \leq 
C \| \chi_{\widetilde{B}} \|_{X}^{p_0}.
\]
Thus,
\[
\| f \|_{H^{X, p_0, q, d}_{fin}(w)} \leq C,
\]
as we wished to prove.
\end{proof}

\begin{lemma} \label{T atom estim}
Let $X$ be a ball quasi-Banach function space and let $0 < p_0 \leq 1$. For $0 \leq \alpha < n$ and 
$m \in \mathbb{N} \cap \left(1 - \frac{\alpha}{n}, +\infty \right)$, let $T_{\alpha, m}$ be the operator defined by (\ref{T}) and 
let $a(\cdot)$ be a $(X, q/p_0, d)$-atom supported on a ball $B$.\\
\textbf{a)} If $0 < \alpha < n$, $w \in \mathcal{A}_1$ and $q > \frac{n p_0}{\alpha}$, then for $\frac{1}{q_0} = \frac{1}{p_0} - \frac{\alpha}{n}$
\[
\int_{\mathbb{R}^{n}} |T_{\alpha, m} a (x)|^{q_0} w(x) dx \leq C |B|^{\frac{\alpha}{n}q_0} \| \chi_{B} \|_{X}^{-q_0} \sum_{i=1}^{m} 
w_{A_{i}^{-1}}(B),
\]
\textbf{b)} If $\alpha=0$ and $w \in \mathcal{A}_1 \cap RH_{(q/p_0)'}$, then
\[
\int_{\mathbb{R}^{n}} |T_{0, m} a (x)|^{p_0} w(x) dx \leq C \| \chi_{B} \|_{X}^{-p_0} \sum_{i=1}^{m} w_{A_{i}^{-1}}(B),
\]
\end{lemma}

\begin{proof}
The proof is analogous to that of \cite[Lemma 3.2]{Rocha2018}, but now considering the $X$-atoms in Definition \ref{atom}.
\end{proof}

We end this section with the following proposition, which is crucial to get our main result presented in the next section.

\begin{proposition} \label{w ineq T}
Let $0 \leq \alpha < n$, $0 < p_0 < \frac{n}{n+ \alpha}$ and let $X$ be a $\mathcal{O}(n)$-invariant ball quasi-Banach function space 
such that $X^{1/p_0}$ is a ball Banach function space. For the operator $T_{\alpha, m}$ defined by (\ref{T}) with 
$m \in \mathbb{N} \cap \left(1 - \frac{\alpha}{n}, +\infty \right)$, we have that: 
\\
\textbf{a)} If $0 < \alpha < n$, $\frac{1}{q_0} := \frac{1}{p_0} - \frac{\alpha}{n}$ and $w \in \mathcal{A}_1 \cap ((X^{1/p_0})')^{p_0/q_0}$, then
\[
\| T_{\alpha, m} f \|_{L^{q_0}(w)} \leq C \sum_{i=1}^{m} \| f \|_{H^{p_0}\left( [w_{A_i^{-1}}]^{\frac{p_{0}}{q_{0}}} \right)},
\]
for all $f \in H^{X, p_0, q/p_0, d}_{fin}(w)$, where $q$ is sufficiently large.
\\
\textbf{b)} If $\alpha=0$ and $w \in \mathcal{A}_1 \cap RH_{(q/p_0)'} \cap (X^{1/p_0})'$, then
\[
\| T_{0, m} f \|_{L^{p_0}(w)} \leq C \sum_{i=1}^{m} \| f \|_{H^{p_0}\left( [w_{A_i^{-1}}] \right)},
\]
for all $f \in H^{X, p_0, q/p_0, d}_{fin}(w)$, where $q$ is sufficiently large.
\end{proposition}

\begin{proof}
Let $0 < p_0 < \frac{n}{n+ \alpha}$ and $\frac{1}{q_0} := \frac{1}{p_0} - \frac{\alpha}{n}$, where $0 \leq \alpha < n$. Given 
$f \in H^{X, p_0, q/p_0, d}_{fin}(w)$ we have that $f = \sum_{j=1}^{k} \lambda_j a_j$, where $a_j$ is an $(X, q/p_0, d)$-atom supported 
on a ball $B_j$. Being $0 < q_0 < 1$, to apply Lemma \ref{T atom estim}, according to the case $0 < \alpha < n$ or $\alpha = 0$, we obtain
\[
\| T_{\alpha, m} f \|_{L^{q_0}(w)}^{q_0} \leq \sum_{j=1}^{k} \lambda_{j}^{q_0} \int_{\mathbb{R}^{n}} |T_{\alpha, m}a_j(x)|^{q_0} w(x) dx
\]
\[
\leq C \sum_{i=1}^{m} \sum_{j=1}^{k} \left(\frac{\lambda_{j}  |B_j|^{\frac{\alpha}{n}}}{\| \chi_{B_j} \|_{X}}\right)^{q_0} 
w_{A_{i}^{-1}}(B_j)
\]
\[
= C \sum_{i=1}^{m} \int_{\mathbb{R}^{n}} \left\{ \sum_{j=1}^{k} \left(\frac{\lambda_{j}  |B_j|^{\frac{\alpha}{n}} \chi_{B_j}(x)}{\| \chi_{B_j} \|_{X}}\right)^{q_0} \right\} w_{A_{i}^{-1}}(x) dx
\]
the embedding $l^{p_0} \hookrightarrow l^{q_0}$ gives
\begin{equation}
\leq C \sum_{i=1}^{m} \int_{\mathbb{R}^{n}} \left\{\sum_{j=1}^{k} \left(\frac{\lambda_{j}  |B_j|^{\frac{\alpha}{n}} \chi_{B_j}(x)}{\| \chi_{B_j} \|_{X}}\right)^{p_0} \right\}^{\frac{q_0}{p_0}} w_{A_{i}^{-1}}(x) dx \label{desig}.
\end{equation}
Now, it is clear that if $\alpha=0$, then the proposition follows from Lemma \ref{HX finito 2}, since 
$w_{A^{-1}_i} \in \mathcal{A}_1 \cap RH_{(q/p_0)'} \cap (X^{1/p_0})'$ (see Remarks \ref{On 1}, \ref{On 2} and \ref{O de n}) and 
$H^{X, p_0, q/p_0, d}_{fin}(w) = H^{X, p_0, q/p_0, d}_{fin}(w_{A_i^{-1}})$ as sets. 

For the case $0 < \alpha < n$, a computation gives 
$|B_j|^{\frac{\alpha}{n}} \chi_{B_j} \leq \left(M_{\frac{\alpha p_0}{2}} (\chi_{B_j})\right)^{\frac{2}{p_0}}$, so (\ref{desig})
\[
\leq C \sum_{i=1}^{m} \int_{\mathbb{R}^{n}} \left\{\sum_{j=1}^{k} \left( \frac{\lambda_{j} 
\left(M_{\frac{\alpha p_0}{2}} (\chi_{B_j})\right)^{\frac{2}{p_0}}}{\| \chi_{B_j} \|_{X}}\right)^{p_0} \right\}^{\frac{q_0}{p_0}} 
w_{A_{i}^{-1}}(x) dx
\]
\[
= C \sum_{i=1}^{m} \left\| \left\{\sum_{j=1}^{k} \frac{\lambda_{j}^{p_0}  \left(M_{\frac{\alpha p_0}{2}} (\chi_{B_j})(\cdot)\right)^{2}}{\| \chi_{B_j} \|_{X}^{p_0}} \right\}^{\frac{1}{2}} \right\|^{\frac{2q_0}{p_0}}_{L^{\frac{2q_0}{p_0}}(w_{A_{i}^{-1}})}
\]
because $\frac{p_0}{2q_0} = \frac{1}{2} - \frac{\alpha p_0}{2n}$ 
and $[w_{A_i^{-1}}]^{\frac{p_0}{2q_0}} \in \mathcal{A}_{2, \frac{2q_0}{p_0}}$, by \cite[Lemma 2.13]{Rocha2018} we have
\[
\leq C \sum_{i=1}^{m} \left\| \left\{\sum_{j=1}^{k} \frac{\lambda_{j}^{p_0}  \chi_{B_j}}{\| \chi_{B_j} \|_{X}^{p_0}} \right\}^{\frac{1}{2}} \right\|^{\frac{2q_0}{p_0}}_{L^{2}\left([w_{A_{i}^{-1}}]^{\frac{p_0}{q_0}} \right)}
\]
\[
=C \sum_{i=1}^{m} \left\| \sum_{j=1}^{k} \frac{\lambda_{j}^{p_0}  \chi_{B_j}}{\| \chi_{B_j} \|_{X}^{p_0}}  \right\|^{\frac{q_0}{p_0}}_{L^{1}\left([w_{A_{i}^{-1}}]^{\frac{p_0}{q_0}} \right)}.
\]
Since, for each $i=1, ..., m$, $[w_{A_i^{-1}}]^{\frac{p_0}{q_0}} \in \mathcal{A}_1 \cap (X^{1/p_0})'$ (see Remarks \ref{On 1}, \ref{On 2}, 
\ref{wr A1} and \ref{O de n}), and $H^{X, p_0, q/p_0, d}_{fin}(w) = H^{X, p_0, q/p_0, d}_{fin}\left([w_{A_i^{-1}}]^{\frac{p_0}{q_0}} \right)$ as sets, once again by Lemma \ref{HX finito 2}, we can take the infimum over all such decompositions to get
\[
\| T_{\alpha, m} f \|_{L^{q_0}(w)} \leq C \sum_{i=1}^{m} \| f \|_{H^{p_0}\left( [w_{A_i^{-1}}]^{\frac{p_{0}}{q_{0}}} \right)},
\]
for all $f \in H^{X, p_0, q/p_0, d}_{fin}(w)$.
\end{proof}

\section{Main result} \label{resultados principales}

In this section, we prove the main result of this work. Our two main techniques are the finite atomic decomposition and weighted norm inequalities of the operator $T_{\alpha, m}$ established in Section \ref{weighted estimates}.

\begin{theorem} \label{main thm}
Given $0 \leq \alpha < n$ and $m \in \mathbb{N} \cap \left(1 - \frac{\alpha}{n}, +\infty \right)$, let $T_{\alpha, m}$ be the fractional operator defined by (\ref{T}) and let $X$ and $Y$ be ball quasi-Banach function spaces such that the quasi-norm of $X$ is $\mathcal{O}(n)$-invariant and absolutely continuous satisfying (\ref{A1}) and (\ref{A2}), and $X$ is strictly $s$-convex, where $s \in (0,1]$ is as in 
(\ref{A1}). 
\\
\textbf{a)} If for some $0 < p_0 < q_0 < 1$ such that $\frac{1}{p_0} - \frac{1}{q_0} = \frac{\alpha}{n}$ with $0 < \alpha < n$, $X$ 
satisfy (\ref{A2}) with $q > \max \left\{1, \frac{p_0 n}{\alpha} \right\}$, $X^{1/p_0}$ and $Y^{1/q_0}$ are ball Banach function spaces such that $(Y^{1/q_0})' = ((X^{1/p_0})')^{p_0/q_0}$ and the Hardy-Littlewood operator $M$ is bounded on $(Y^{1/q_0})'$,
then $T_{\alpha, m}$ can be extended to a bounded operator $H_X(\mathbb{R}^n) \to Y$.
\\
\textbf{b)} If for some $0 < p_0 \leq 1$, $X^{1/p_0}$ is a ball Banach function space, the Hardy-Littlewood operator $M$ is bounded on 
$(X^{1/p_0})'$ and $X$ satisfy (\ref{A2}) with 
\begin{equation} \label{qb}
q > \max \left\{1, p_0 \left(1+2^{n+3} \| M \|_{(X^{1/p_0})'} \right) \right\},
\end{equation}
then $T_{0, m}$ can be extended to a bounded operator $H_X(\mathbb{R}^n) \to X$.
\end{theorem}

\begin{proof} The operator $T_{\alpha, m}$ is well defined on the elements of $H^{X, q/p_0, d}_{fin}(\mathbb{R}^{n})$. So given 
$f \in H^{X, q/p_0, d}_{fin}(\mathbb{R}^{n})$, from Lemma \ref{doble dual}, we have that
\[
\| T_{\alpha, m} f\|_{Y}^{q_0} = \| |T_{\alpha, m} f|^{q_0} \|_{Y^{1/q_0}} = \sup \int |T_{\alpha, m} f(x)|^{q_0} |g (x)| dx,
\]
where the supremum is taken over all $g \in (Y^{1/q_0})'$ such that $\| g \|_{(Y^{1/q_0})'} \leq 1$.
Now we utilize the Rubio de Francia iteration algorithm with respect to $(Y^{1/q_0})'$. Given a function $g$, define
\[
\mathcal{R}g (x) = \sum_{i=0}^{\infty} \frac{M^{i}g(x)}{2^{i} \|M\|_{(Y^{1/q_0})'}^{i}},
\]
where $M^{0}g = g$ and, for $i \geq 1$, $M^{i}g = M \circ \cdot \cdot \cdot \circ M g$ denotes $i$ iterates of the Hardy-Littlewood maximal operator. The function $\mathcal{R}g$ satisfies:

$(1) \, |g(x)| \leq \mathcal{R}g(x)$ for all $x \in \mathbb{R}^{n}$;

$(2) \, \| \mathcal{R}g \|_{(Y^{1/q_0})'} \leq C \|g \|_{(Y^{1/q_0})'}$;

$(3) \, \mathcal{R}g \in \mathcal{A}_1$ and $[\mathcal{R}g]_{\mathcal{A}_1} \leq 2 \| M\|_{(Y^{1/q_0})'}$.\\
by these properties and since $(Y^{1/q_0})' = ((X^{1/p_0})')^{p_0/q_0}$ and 
$H^{X, q/p_0, d}_{fin}(\mathbb{R}^{n}) = H^{p_0, q/p_0, d}_{fin}(\mathcal{R}g)$ as sets, Proposition \ref{w ineq T} - \textbf{a)} gives
\[
\int |T_{\alpha, m} f(x)|^{q_0} |g (x)| dx \leq \int |T_{\alpha, m} f(x)|^{q_0} \mathcal{R}g(x) dx
\]
\[
= \| T_{\alpha, m}f \|_{L^{q_0}(\mathcal{R}g)}^{q_0}
\leq C \sum_{i=1}^{m} \|f \|_{H^{p_0}\left([(\mathcal{R}g)_{A_{i}^{-1}}]^{\frac{p_0}{q_0}} \right)}^{q_0}
\]
\[
=C  \sum_{i=1}^{m} \left( \int [\mathcal{M}_{N}f(x)]^{p_0} [(\mathcal{R}g)_{A_{i}^{-1}}(x)]^{\frac{p_0}{q_0}} dx \right)^{\frac{q_0}{p_0}}
\]
by H\"older's inequality (see \cite[Lemma 2.5]{Zhang}), we have
\[
\leq C \| [\mathcal{M}_{N}f]^{p_0} \|_{X^{1/p_0}}^{\frac{q_0}{p_0}} \sum_{i=1}^{m}\left\|[(\mathcal{R}g)_{A_{i}^{-1}}]^{\frac{p_0}{q_0}} \right\|_{(X^{1/p_0})'}^{\frac{q_0}{p_0}}
\]
\[
= C \| \mathcal{M}_{N}f \|_{X}^{q_0} \sum_{i=1}^{m}\left\|[(\mathcal{R}g)_{A_{i}^{-1}}] 
\right\|_{((X^{1/p_0})')^{p_0/q_0}}
\]
since $X$ is $\mathcal{O}(n)$-invariant (see Remarks \ref{On 1} and \ref{On 2}) and $(Y^{1/q_0})' = ((X^{1/p_0})')^{p_0/q_0}$, we have
\[
= C \| f \|_{H_X(\mathbb{R}^n)}^{q_0} \sum_{i=1}^{m} \|\mathcal{R}g\|_{(Y^{1/q_0})'}
\]
\[
\leq  C \|f \|_{H_X(\mathbb{R}^n)}^{q_0} \left\|g \right\|_{(Y^{1/q_0})'}.
\]
Thus
\[
\| T_{\alpha, m} f\|_{Y} \leq C \|f \|_{H_X(\mathbb{R}^n)},
\]
for all $f \in H^{X, q/p_0, d}_{fin}(\mathbb{R}^{n})$. Then, the part \textbf{a)} of the theorem follows from Corollary \ref{dense set}. 

Finally, to prove the part \textbf{b)} of the theorem (i.e.: the case $\alpha = 0$), we proceed as in the proof of \textbf{a)}, by considering $p_0=q_0$, $X = Y$, Proposition \ref{w ineq T} - \textbf{b)} and verifying that $Rg \in \mathcal RH_{(q/p_0)'}$. Indeed, (\ref{qb}), Remark \ref{RHs RHt} and Lemma \ref{A1 y RHs} give $\mathcal{R}g \in RH_{(q/p_0)'}$. Thus \textbf{b)} follows and with them, the theorem.
\end{proof}

\begin{corollary} \label{estim riesz}
Let $0 < \alpha < n$ and $0 < p_0 < q_0 \leq 1$ and let $X$ and $Y$ be ball quasi-Banach function spaces as in Theorem \ref{main thm}. Then the Riesz potential $I_{\alpha}$ given by (\ref{Riesz pot}) can be extended to a bounded operator $H_X(\mathbb{R}^n) \to Y$.
\end{corollary}

\begin{proof}
Apply Theorem \ref{main thm} - \textbf{a)} with $m=1$ and $A_1 = I$.
\end{proof}

\section{Applications} \label{appl}

In this section we illustrate our results with four concrete examples of ball quasi-Banach function spaces 
$X$ and $Y$ satisfying the hypotheses of Theorem \ref{main thm}. That is:

(i) $X$ satisfies (\ref{A1}) and (\ref{A2}).

(ii) $X$ is strictly $s$-convex, where $s \in (0,1]$ is as in (\ref{A1}). 

(iii) $\| \cdot \|_X$ is $\mathcal{O}(n)$-invariant and absolutely continuous.

(iv - \textbf{a}) For $0 < \alpha < n$, there exist $0 < p_0 < q_0 \leq 1$ such that $\frac{1}{p_0} - \frac{1}{q_0} = \frac{\alpha}{n}$, $X$ 
satisfy (\ref{A2}) with $q > \max \left\{1, \frac{p_0 n}{\alpha} \right\}$, $X^{1/p_0}$ and $Y^{1/q_0}$ are ball Banach function spaces such that $(Y^{1/q_0})' = ((X^{1/p_0})')^{p_0/q_0}$ and the Hardy-Littlewood operator $M$ is bounded on $(Y^{1/q_0})'$.

(iv - \textbf{b}) There exists $0 < p_0 \leq 1$ such that $X^{1/p_0}$ is a ball Banach function space, the Hardy-Littlewood operator $M$ is bounded on $(X^{1/p_0})'$ and $X$ satisfy (\ref{A2}) with $q > \max \left\{1, p_0 \left(1+2^{n+3} \| M \|_{(X^{1/p_0})'} \right) \right\}$.

\

{\bf Weighted Lebesgue spaces.} Given $0 < p < \infty$ and a weight $w \in \mathcal{A}_{\infty}$ (see \cite{grafakos}, \cite{Stein}), the weighted Lebesgue space $L^{p}(w)$ is defined as the set of all the measurable functions $f$ on $\mathbb{R}^n$ such that
\[
\| f \|_{L^{p}(w)} := \left(  \int_{\mathbb{R}^n} |f(x)|^p w(x) dx \right)^{1/p} < \infty.
\]
By \cite[Section 7.1]{sawa}, the couple $(L^{p}(w), \| \cdot \|_{L^{p}(w)})$ is a ball quasi-Banach function space. If $p > 1$, then 
$L^p(w)$ is a ball Banach function space with $(L^p(w))' = L^{p'}(w^{1-p'})$, where $\frac{1}{p} + \frac{1}{p'} = 1$. 
In \cite[Section 7.1]{sawa}, it was pointed out that a weighted Lebesgue space may not be a Banach function space. Given $0 \leq \alpha < n$, let $\frac{1}{q} = \frac{1}{p} - \frac{\alpha}{n}$ with $p \in (0, \frac{n}{\alpha})$. For such $p$ and $q$, we consider $X=L^p(w)$ and 
$Y = L^q(w^{q/p})$ with $w(x) = |x|^{a}$ and $-n < a \leq 0$. Next, we will verify the hypotheses of Theorem \ref{main thm} for such $X$ and 
$Y$.

(i) For $X = L^p(w)$, the assumption (\ref{A1}) holds true for any 
$\theta, s \in (0,1]$, $\theta < s$, $p \in (\theta, \frac{n}{\alpha})$ and $w \in \mathcal{A}_{p/\theta}$ (it follows from 
\cite[Theorem 3.1 (b)]{John} applied on $L^{p/\theta}(w)$, with $r = s/\theta$ and $|f_j|^{\theta}$ instead of $f_k$). 
The assumption (\ref{A2}), by duality, \cite[Theorem 9]{Muckenh} and Remark \ref{cond equiv A2}, holds true for any 
$r \in (0, \min\{ 1, p \})$, $w \in \mathcal{A}_{p/r}$, and $\widetilde{q} \in (\max\{1, p\}, \infty)$ sufficiently large such that 
$w^{1-(p/r)'} \in \mathcal{A}_{(p/r)'/(\widetilde{q}/r)'}$.

(ii) By Fatou's Lemma, it follows that for any $s \in (0,p)$, the space $L^{p}(w)$ is \textit{strictly $s$-convex}, where 
$(L^{p}(w))^{1/s} = L^{p/s}(w)$ .

(iii) In general, if $w \in \mathcal{A}_{\infty}$, by the dominated convergence Theorem, we have that the quasi-norm of $X = L^{p}(w)$ is \textit{absolutely continuous}. Now, for $w(x) = |x|^{a}$ with $-n < a \leq 0$, it is clear that $X$ is $\mathcal{O}(n)$-\textit{invariant}. 

(iv - \textbf{a} and \textbf{b}) For any $s \in (0, p)$ and $w \in \mathcal{A}_{p/s}$, the Hardy-Littlewood maximal operator $M$ is 
bounded on $L^{p/s}_w$ and so also on $(L^{p/s}_{w})' = L^{(p/s)'}_{w^{1-(p/s)'}}$ since $w^{1-(p/s)'} \in \mathcal{A}_{(p/s)'}$ 
(see \cite[Theorem 9]{Muckenh}).  Let $Y = L^q(w^{q/p})$ with $0 <  p < \frac{n}{\alpha}$, $\frac{1}{q} = \frac{1}{p} - \frac{\alpha}{n}$ and
$0 \leq \alpha < n$. For any $p_0 \in (0, \min\{ \frac{n}{n+\alpha}, p \})$ fixed, we put $\frac{1}{q_0} := \frac{1}{p_0} - \frac{\alpha}{n}$, then $0 < p_0 \leq q_0 \leq 1$, $q_0 \in (0, q)$, $X^{1/p_0}$ and $Y^{1/q_0}$ are ball Banach function spaces and $M$ is bounded on 
$(Y^{1/q_0})'$. Since $\frac{1}{p} - \frac{1}{q} = \frac{\alpha}{n} = \frac{1}{p_0} - \frac{1}{q_0}$, it follows that 
$(Y^{1/q_0})' = ((X^{1/p_0})')^{p_0/q_0}$. By (i), the $\widetilde{q}$ satisfying (\ref{A2}) can be chosen such that 
$\widetilde{q} > \max \left\{1, \frac{p_0 n}{\alpha} \right\}$ or
$\widetilde{q} > \max \left\{1, p_0 \left(1+2^{n+3} \| M \|_{(X^{1/p_0})'} \right) \right\}$, according to the case.

Finally, since $w(x) = |x|^a \in \mathcal{A}_1$ for $-n < a \leq n$ (see \cite[Example 7.1.7]{grafakos}) and 
$\mathcal{A}_1 \subset \mathcal{A}_p$ for all $p \geq 1$, Theorem \ref{main thm} (and so also Corollary \ref{estim riesz}) applies with 
$X=L^p(|\cdot|^a)$ and $Y = L^q(|\cdot|^{aq/p})$, where $0 < p < \frac{n}{\alpha}$, $\frac{1}{q} = \frac{1}{p} - \frac{\alpha}{n}$, $0 \leq \alpha < n$ and $-n < a \leq 0$ (cf. \cite[Theorems 17 and 19]{Rocha2023}).

\

{\bf Variable Lebesgue spaces.} Let $p(\cdot) : \mathbb{R}^n \to (0, \infty)$ be a measurable function such that
\[ 
0 < p_{-} := \essinf_{x \in \mathbb{R}^n} p(x) \leq \esssup_{x \in \mathbb{R}^n} p(x) : = p_{+} < \infty.
\]
Define the modular associated with $p(\cdot)$ by
\[
\kappa_{p(\cdot)}(f) := \int_{\mathbb{R}^n} |f(x)|^{p(x)} dx.
\]
Then, the variable Lebesgue space $L^{p(\cdot)}(\mathbb{R}^n)$ consists of all the measurable functions $f$ such that
\[
\| f \|_{p(\cdot)} := \inf \left\{ \lambda > 0 : \kappa_{p(\cdot)}(f/\lambda) \leq 1 \right\} < \infty.
\]
The amount $\Vert f \Vert_{p(\cdot)}$ is known as the Luxemburg norm of $f$ with respect to $p(\cdot)$. By \cite[Section 7.8]{sawa}, we have that the couple $(L^{p(\cdot)}(\mathbb{R}^n), \| \cdot \|_{p(\cdot)})$ is a ball quasi-Banach function space. If $p_{-} > 1$, then $L^{p(\cdot)}(\mathbb{R}^n)$ is a ball Banach function space with $(L^{p(\cdot)}(\mathbb{R}^n))' = L^{p'(\cdot)}(\mathbb{R}^n)$, where $\frac{1}{p(x)} + \frac{1}{p'(x)} = 1$ for all $x$. An exponent $p(\cdot) : \mathbb{R}^n \to (0, \infty)$ is said to be globally log-H\"older continuous if there exist positive constants $C$ and $p_{\infty}$ such that
\[
|p(x) - p(y)| \leq \frac{-C}{\log(|x-y|)}, \,\, \text{for} \,\, |x-y| \leq 1/2
\]
and
\[
|p(x) - p_{\infty}| \leq \frac{C}{\log(e + |x|)}, \,\, \text{for all} \,\, x \in \mathbb{R}^n.
\]
From now on, we consider the space $L^{p(\cdot)}(\mathbb{R}^n)$ with $p(\cdot)$ globally log-H\"older continuous.

(i) For $X:= L^{p(\cdot)}$, the assumption (\ref{A1}) holds true for any $s \in (0,1]$ and $\theta \in (0, \min\{s, p_{-}\})$ (indeed, we apply \cite[Lemma 2.4]{Nakai} on $L^{p(\cdot)/\theta}$, with $u=s/\theta$ and $|f_j|^{\theta}$ instead of $f_j$); the assumption (\ref{A2}) holds true for any $r \in (0, \min\{ 1, p_{-}\})$ and $\widetilde{q} \in (\max \{1, p_{+}\}, \infty)$ (this follows by duality, 
\cite[Theorem 1.5]{DCruz} and Remark \ref{cond equiv A2}).

(ii) By \cite[Proposition 2.18 and Theorem 2.61]{Fiorenza}, for any $s \in (0, p_{-})$, $(L^{p(\cdot)})^{1/s} = L^{p(\cdot)/s}$ and the space $L^{p(\cdot)}$ is \textit{strictly $s$-convex}.

(iii) By dominated convergence Theorem, we have that if $\{ f_j \} \subset X$ and $f_j \to 0$ a.e., then $\kappa_{p(\cdot)}(f_j) \to 0$. 
So, from \cite[Theorem 2.69]{Fiorenza}, it follows $\| f_j \|_{p(\cdot)} \to 0$. Then, the quasi-norm $\| \cdot \|_{p(\cdot)}$ is 
\textit{absolutely continuous}. If the exponent $p(\cdot)$ is radial (i.e.: for any $A \in \mathcal{O}(n)$ fixed, $p(Ax) = p(x)$ for all 
$x \in \mathbb{R}^n$), then the quasi-norm 
$\| \cdot \|_{p(\cdot)}$ is $\mathcal{O}(n)$-\textit{invariant}.

(iv - \textbf{a} and \textbf{b}) For any $s \in (0, p_{-})$, the Hardy-Littlewood maximal operator $M$ is bounded on $L^{p(\cdot)/s}$ (see \cite[Theorem 1.5]{DCruz}), and so also on $(L^{p(\cdot)/s})' = L^{(p(\cdot)/s)'}$, since the exponent $(p(\cdot)/s)'$ results globally 
log-H\"older continuous with $((p(\cdot)/s)')_{-} > 1$. Given $0 \leq \alpha < n$, let $p(\cdot)$ be an exponent such that $0 < p_{-} \leq p_{+} < \frac{n}{\alpha}$ and is globally log-H\"older continuous. Then, we define $\frac{1}{q(\cdot)} : = \frac{1}{p(\cdot)} - \frac{\alpha}{n}$, 
such $q(\cdot)$ results globally log-H\"older continuous. Let $X:= L^{p(\cdot)}(\mathbb{R}^n)$ and $Y := L^{q(\cdot)}(\mathbb{R}^n)$. For any 
$p_0 \in (0, \min\{ \frac{n}{n+\alpha}, p_{-} \})$ fixed, we put $\frac{1}{q_0} := \frac{1}{p_0} - \frac{\alpha}{n}$, then 
$0 < p_0 \leq q_0 \leq 1$, $q_0 \in (0, q_{-})$, $X^{1/p_0}$ and $Y^{1/q_0}$ are ball Banach function spaces and $M$ is 
bounded on $(Y^{1/q_0})'$. Since $\frac{1}{p(\cdot)} - \frac{1}{q(\cdot)} = \frac{\alpha}{n} = \frac{1}{p_0} - \frac{1}{q_0}$ it follows that
$(Y^{1/q_0})' = ((X^{1/p_0})')^{p_0/q_0}$. By (i), the $\widetilde{q}$ satisfying (\ref{A2}) can be chosen such that 
$\widetilde{q} > \max \left\{1, \frac{p_0 n}{\alpha} \right\}$ or
$\widetilde{q} > \max \left\{1, p_0 \left(1+2^{n+3} \| M \|_{(X^{1/p_0})'} \right) \right\}$, according to the case.

Now, we have that $H_{X}(\mathbb{R}^n) = H^{p(\cdot)}(\mathbb{R}^n)$ is the variable Hardy spaces with exponent $p(\cdot)$ defined in 
\cite{Nakai}. Finally, if $p(\cdot)$ is radial and globally log-H\"older continuous, then Theorem \ref{main thm} 
(and so also Corollary \ref{estim riesz}) applies with $X= L^{p(\cdot)}$ and $Y = L^{q(\cdot)}$ and recover \cite[Theorem 1.1]{Rocha2018}.

\

{\bf Lorentz spaces.} Given $0 < p, \, q < \infty$, the Lorentz space $L^{p, q}(\mathbb{R}^n)$ is defined as the collection of all the measurable function $f$ such that
\[
\| f \|_{L^{p, q}} := \left( \int_{0}^{\infty} (t^{1/p} f^{\ast}(t))^{q} \, \frac{dt}{t} \right)^{1/q} 
< \infty,
\]
where $f^{\ast}$, the decreasing rearrangement function of $f$, is defined by setting, for any $t \in [0, \infty)$, 
$f^{\ast}(t) := \inf \{ s > 0 : |\{x : |f(x)| > s \}| \leq t \}$.
By \cite[Section 7.3]{sawa}, the couple $\left( L^{p, q}(\mathbb{R}^n), \| \cdot \|_{L^{p, q}} \right)$  is a ball quasi-Banach function space, whose quasi-norm $\| \cdot \|_{L^{p, q}}$ satisfies $\| |g |^r \|_{L^{p, q}} = \| g \|_{L^{pr, qr}}^r$ for all 
$0 < p, \, q, r < \infty$. When $1 < p, \, q < \infty$, $(L^{p, q})' = L^{p', q'}$, where $\frac{1}{p} + \frac{1}{p'} = 1$ and 
$\frac{1}{q} + \frac{1}{q'} = 1$ (see \cite[Theorem 4.7]{Bennett}).

(i) For $X:= L^{p, q}(\mathbb{R}^n)$, the assumption (\ref{A1}) holds true for any $s \in (0,1]$ and $\theta \in (0, \min\{ s, p, q \})$ 
(apply conveniently \cite[Theorem 2.3 (iii)]{Curbera}). Now, by duality, \cite[Theorem 2.3 (iii)]{Curbera} and Remark \ref{cond equiv A2}, 
the assumption (\ref{A2}) holds true for any $r \in (0, \min\{ 1, p, q \})$ and $\widetilde{q} \in (\max \{1, p, q\}, \infty)$.
 
(ii) For any $s \in (0, \min\{p, q \})$, by \cite[Theorem 4.6]{Bennett}, the space $(L^{p, q})^{1/s} = L^{p/s, q/s}$ is a Banach function space and so $L^{p, q}$ is \textit{strictly $s$-convex} (see \cite[Theorem 1.6]{Bennett}).

(iii) For $0 < p, q < \infty$, it is clear that $\| \cdot \|_{L^{p, q}}$ is $\mathcal{O}(n)$-\textit{invariant} and \textit{absolutely continuous} (see \cite[Chapter 4, Section 4]{Bennett}).

(iv - \textbf{a} and \textbf{b}) Given $0 \leq \alpha < n$ and $0 < p, q < \frac{n}{\alpha}$, we put $\frac{1}{u} := \frac{1}{p} - \frac{\alpha}{n}$ and $\frac{1}{v} := \frac{1}{q} - \frac{\alpha}{n}$. Now, we consider $X:= L^{p, q}(\mathbb{R}^n)$ and 
$Y:= L^{u, v}(\mathbb{R}^n)$. For any $p_0 \in (0, \min\{ \frac{n}{n+\alpha}, p, q \})$, we put 
$\frac{1}{q_0}:= \frac{1}{p_0} - \frac{\alpha}{n}$, then $0 < p_0 \leq q_0 \leq 1$, $q_0 \in (0, \min\{u, v\})$, $X^{1/p_0}$ 
and $Y^{1/q_0}$ are ball Banach function spaces such that $M$ is bounded on $(Y^{1/q_0})'$ (see \cite[Theorem 2.3 (i)]{Curbera}). Since  
$\frac{1}{p} - \frac{1}{u}= \frac{1}{p_0} - \frac{1}{q_0}= \frac{1}{q} - \frac{1}{v}$, we have $(Y^{1/q_0})' = ((X^{1/p_0})')^{p_0/q_0}$.
By (i), the $\widetilde{q}$ in (\ref{A2}) can be chosen such that $\widetilde{q} > \max \left\{1, \frac{p_0 n}{\alpha} \right\}$ or
$\widetilde{q} > \max \left\{1, p_0 \left(1+2^{n+3} \| M \|_{(X^{1/p_0})'} \right) \right\}$, according to the case. 

Finally, $H_X (\mathbb{R}^n) = H^{p, q}(\mathbb{R}^n)$ is the Hardy-Lorentz space defined in \cite{Abu}. Then, Theorem \ref{main thm} 
(and so also Corollary \ref{estim riesz}) applies on such $X$ and $Y$.

\

{\bf Orlicz spaces.} A function $\Phi : [0, \infty) \to [0, \infty)$ is called an \textit{Orlicz function} if

(1) it is non-decreasing and satisfies $\lim_{t \to 0^{+}} \Phi(t) = \Phi(0) = 0$, $\Phi(t) > 0$ for all $t>0$ and 
$\lim_{t \to \infty} \Phi(t) = \infty$;

(2) the function $x \to \Phi(\vert f(x) \vert)$ is measurable for every measurable function $f$ on $\mathbb{R}^n$.

\

Given an Orlicz function $\Phi$, define the modular associated with $\Phi$ by
\[
\kappa_{\Phi}(f) := \int_{\mathbb{R}^n} \Phi(|f(x)|) \, dx.
\]
Then, the Orlicz space $L^{\Phi}(\mathbb{R}^n)$ is defined to be the set of all measurable functions $f$ on $\mathbb{R}^n$ such that
\[
\| f \|_{\Phi} := \inf \left\{ \lambda > 0 : \kappa_{\Phi}(f/\lambda) \leq 1 \right\} < \infty. 
\]
The amount $\| f \|_{\Phi}$ is known as the Luxemburg norm of $f$ with respect to $\Phi$. We say that two Orlicz function $\Phi$ and $\Psi$ are equivalent, denote $\Phi \sim \Psi$, if $\Psi(t/C) \leq \Phi(t) \leq \Psi(Ct)$ for some $C \geq 1$ and all $t \geq 0$, and so 
$L^{\Phi} = L^{\Psi}$ with $\| \cdot \|_{\Phi} \approx  \| \cdot \|_{\Psi}$. For any $s \in (0, \infty)$ fixed, we set 
$\Phi_s(t) := \Phi(t^s)$ for all $t >0$. Then, $\Phi_s$ is an Orlicz function and for any measurable function $f$, one has 
$\| f \|_{\Phi}^{s} = \| |f|^s \|_{\Phi_{1/s}}$.

An Orlicz function $\Phi : [0, \infty) \to [0, \infty)$ is said to be of positive lower (respectively, positive upper) type $p$ with 
$p \in (0, \infty)$ if there exists a positive constant $C$, depending on $p$, such that, for any $t >0$ and $r \in (0,1]$ (respectively, 
$r \in [1, \infty)$),
\begin{equation} \label{lower upper type}
\Phi(rt) \leq C r^p \Phi(t).
\end{equation}
When $\Phi$ is an Orlicz function with positive lower type $p_{\Phi}^{-}$ and positive upper type $p_{\Phi}^{+}$, by \cite[p. 92]{sawa}, the couple $(L^{\Phi}, \| \cdot \|_{\Phi})$ is a ball quasi-Banach space. Moreover, if $\Phi$ is an Orlicz function of positive lower (resp., positive upper) type $p_{\Phi}^{-}$ (resp., type $p_{\Phi}^{+}$),  then $\Phi_s$ is of positive lower (resp., positive upper) type  
$s p_{\Phi}^{-}$ (resp., type $s p_{\Phi}^{+}$).

A function $\Phi : [0, \infty) \to [0, \infty)$ is called a \textit{Young function} if $\Phi$ is convex, left-continuous, $\lim_{t \to 0^{+}} \Phi(t) = \Phi(0) = 0$, and $\lim_{t \to \infty} \Phi(t) = \infty$. From the convexity and $\Phi(0) = 0$, it follows that any Young function is non-decreasing.

If $\Phi$ is a young function, from \cite[Lemma 3.2.2 (b) and Theorem 3.3.7 (b)]{Hasto}, it follows that the couple 
$(L^{\Phi}, \| \cdot \|_{\Phi})$ is a Banach space.

If an Orlicz function $\Phi$ is also a Young function, then $\Phi$ is a bijective continuous function from $[0, \infty)$ onto $[0, \infty)$.

For a Young function $\Phi$, we define $\Phi^{-1}$ and its complementary function $\widetilde{\Phi}$ on $[0, \infty)$ by
\[
\Phi^{-1}(s):= \inf\{ t \geq 0 : \Phi(t) > s \}
\]
and
\[
\widetilde{\Phi}(t) := \sup\{ ts - \Phi(s) : s \in [0, \infty) \},
\]
respectively. Then, the K\"othe dual $(L^{\Phi})' = L^{\widetilde{\Phi}}$ with comparable norms 
(see \cite[Lemma 2.4.2 and Theorem 3.4.6]{Hasto}). Moreover, by \cite[Property 1.6]{ONeil}, we have that
\begin{equation} \label{prop 1.6}
s \leq \Phi^{-1}(s) (\widetilde{\Phi})^{-1}(s) \leq 2s, \,\,\,\, s \geq 0.
\end{equation}
If $\Phi$ is a Young-Orlicz function, then $\Phi^{-1}$ is the usual inverse function of $\Phi$.

An Orlicz function $\Phi$ is called an $N$-function if it is a continuous and convex function such that
\begin{equation} \label{2 limites}
\lim_{t \to 0^{+}} \frac{\Phi(t)}{t} = 0, \,\,\,\,\,\, \text{and} \,\,\,\,\,\,  \lim_{t \to \infty} \frac{\Phi(t)}{t} = \infty.
\end{equation}
From \cite[Lemma 2.16]{Zhang2019} we have that if $\Phi$ is an Orlicz function of positive lower type $p_{\Phi}^{-} \in (1, \infty)$ and positive upper type $p_{\Phi}^{+}$, then there exists an Orlicz 
$N$-function $\Psi$ equivalent to $\Phi$. Thus, without loss of generality, we may always assume that an Orlicz function $\Phi$ of positive lower type $p_{\Phi}^{-}$ and positive upper type $p_{\Phi}^{+}$, with $1 < p_{\Phi}^{-} \leq p_{\Phi}^{+} < \infty$, is also an $N$-function. In particular, an Orlicz $N$-function is a Young-Orlicz function.

(i) If $\Phi$ is an Orlicz function with positive lower type $p_{\Phi}^{-}$ and positive upper type 
$p_{\Phi}^{+}$, then (\ref{A1}) holds true for $X= L^{\Phi}(\mathbb{R}^n)$, any $s \in (0,1]$ and any 
$\theta \in (0, \min\{ s, p_{\Phi}^{-}\})$ (see \cite[Theorem 7.14 (i)]{sawa}). To verify (\ref{A2}), let $s \in (0, p_{\Phi}^{-})$. Since, 
we can assume that $\Phi_{1/s}$ is an Orlicz $N$-function, by \cite[Theorem 3.4.6]{Hasto}, we have 
\[
(X^{1/s})' = L^{\widetilde{(\Phi_{1/s})}}, \,\,\,\, \text{and so} \,\,\,\, [(X^{1/s})']^{1/(q/s)'}= L^{\widetilde{(\Phi_{1/s})}_{1/(q/s)'}}
\]
for any $q \in (\max\{1, p_{\Phi}^{+} \}, \infty)$. In turn, by \cite[Lemma 2.4.2]{Hasto} and (\ref{2 limites}), the complementary function 
$\widetilde{(\Phi_{1/s})}$ is a Young-Orlicz function. Moreover, by \cite[Proposition 7.8]{sawa}, the function
$\Theta(t) := \widetilde{(\Phi_{1/s})}_{1/(q/s)'}(t)$ is such that $p_{\Theta}^{-} = (p_{\Phi}^{-}/s)'/(q/s)' > 1$, being $\Theta$ an Orlicz function, by \cite[Theorem 7.12]{sawa}, the Hardy-Littlewood maximal $M$ is bounded on $L^{\Theta}$. Finally, Remark \ref{cond equiv A2} implies that (\ref{A2}) holds true for $X = L^{\Phi}$, any $s \in (0, p_{\Phi}^{-})$ and any $q \in (\max\{1, p_{\Phi}^{+} \}, \infty)$.

(ii) Given an Orlicz function $\Phi$ with positive lower type $p_{\Phi}^{-}$ and positive upper type 
$p_{\Phi}^{+}$, and $s \in (0, p_{\Phi}^{-})$, by \cite[Lemma 2.16]{Zhang2019}, there exists 
a Young-Orlicz function $\Psi$ equivalent to $\Phi_{1/s}$ (i.e. $L^{\Phi_{1/s}} = L^{\Psi}$ with comparable quasi-norms), since 
$\| \cdot \|_{\Psi}$ is a norm, we can assume that $X = L^{\Phi}$ is strictly $s$-convex for $s \in (0, p_{\Phi}^{-})$.

(iii) It is clear that $\Vert \cdot \Vert_{\Phi}$ is $\mathcal{O}(n)$-\textit{invariant}. Since $\lim_{t \to 0^{+}} \Phi(t) = \Phi(0) = 0$, 
by dominated convergence Theorem, we have for every sequence $\{ f_j \} \subset L^{\Phi}$, with $f_j \downarrow 0$, that 
$\kappa_{\Phi}(f_j) \to 0$ and so $\Vert f_j \Vert_{\Phi} \to 0$. Let's see this, suppose that $0< \kappa_{\Phi}(f_j) \to 0$. Given $
0 < \epsilon < 1$ and $C > 0$ as in (\ref{lower upper type}), we have $0 < C \kappa_{\Phi}(f_j) < \epsilon^{p_{\Phi}^{+}}$, for all $j$ large enough, and
\[
\kappa_{\Phi} \left( C^{-1/p_{\Phi}^{+}} \kappa_{\Phi}(f_j)^{-1/p_{\Phi}^{+}} f_j \right)  \leq \kappa_{\Phi}(f_j)^{-1} \kappa_{\Phi}(f_j) = 1, 
\]
so $\| f_j \|_{\Phi} \leq C^{1/p_{\Phi}^{+}} \kappa_{\Phi}(f_j)^{1/p_{\Phi}^{+}} < \epsilon$. Then, $\| \cdot \|_{\Phi}$ is \textit{absolutely continuous}.

(iv \textbf{a} and \textbf{b}).  Given $0 \leq \alpha < n$, let $\Phi$ be an Orlicz function such that 
$0 < p_{\Phi}^{-} \leq p_{\Phi}^{+} < \frac{n}{\alpha}$, by \cite[Lemma 2.5]{Zhang2019} we can assume that $\Phi$ is invertible. Now, let 
$\Psi$ be an Orlicz function such that
\begin{equation} \label{inversas}
\Psi^{-1}(t) \sim t^{-\alpha/n} \Phi^{-1}(t).
\end{equation}
So, such $\Psi$ is of positive lower type $p_{\Psi}^{-} = \frac{n p_{\Phi}^{-}}{n - \alpha p_{\Phi}^{-}}$ and upper type 
$p_{\Psi}^{+} =  \frac{n p_{\Phi}^{+}}{n - \alpha p_{\Phi}^{+}}$.

Now, we consider $X:= L^{\Phi}(\mathbb{R}^n)$ and $Y:= L^{\Psi}(\mathbb{R}^n)$. For any $p_0 \in (0, \min\{ \frac{n}{n+\alpha}, p_{\Phi}^{-} \})$, we put $\frac{1}{q_0}:= \frac{1}{p_0} - \frac{\alpha}{n}$, then $0 < p_0 \leq q_0 \leq 1$, $q_0 \in (0, p_{\Psi}^{-})$, $X^{1/p_0}$ 
and $Y^{1/q_0}$ are ball Banach function spaces such that $M$ is bounded on $(Y^{1/q_0})'$ (see \cite[Theorem 7.12]{sawa}). Since  
$\frac{1}{p_0} - \frac{1}{q_0}= \frac{\alpha}{n}$, (\ref{prop 1.6}) and (\ref{inversas}) give 
$[\widetilde{(\Phi_{1/p_0})}]_{p_0/q_0} \sim \widetilde{(\Psi_{1/q_0})}$, and thus $(Y^{1/q_0})' = ((X^{1/p_0})')^{p_0/q_0}$.
By (i), the $q$ in (\ref{A2}) can be chosen such that $q > \max \left\{1, \frac{p_0 n}{\alpha} \right\}$ or
$q > \max \left\{1, p_0 \left(1+2^{n+3} \| M \|_{(X^{1/p_0})'} \right) \right\}$, according to the case. 

Finally, $H_X (\mathbb{R}^n) = H^{\Phi}(\mathbb{R}^n)$ is the Orlicz-Hardy space defined in \cite{Nakai2014}. Then, Theorem \ref{main thm} (and so also Corollary \ref{estim riesz}) applies on such $X$ and $Y$.


Pablo Rocha, Instituto de Matem\'atica (INMABB), Departamento de Matem\'atica, Universidad Nacional del Sur (UNS)-CONICET, Bah\'ia Blanca, Argentina. \\
{\it e-mail:} pablo.rocha@uns.edu.ar

\end{document}